\documentclass[12pt]{article}
\usepackage[a4paper, total={6in, 8in}, margin={1in}]{geometry}
\usepackage[utf8]{inputenc}

\usepackage{amsmath}
\usepackage{amsthm}
\usepackage{amsfonts}
\usepackage[hidelinks]{hyperref}
\usepackage{scrextend}

\theoremstyle{plain}
\newtheorem{theorem}{Theorem}
\newtheorem{corollary}[theorem]{Corollary}
\newtheorem{lemma}[theorem]{Lemma}

\theoremstyle{definition}

\theoremstyle{remark}

\title{Some Expressions involving Bernoulli, Harmonic, and Stirling Numbers}
\author{Hongshen Chua \\
	}
\date{}

\begin{document}
\maketitle


\bigskip

\section{Introduction}

The Bernoulli numbers, $ B_n $, arose in the study of sums of powers of positive integers. They are commonly defined via the generating function
\begin{align*}
	\frac{z}{e^z - 1} = \sum_{n = 0}^\infty B_n \frac{z^n}{n!}.
\end{align*}
The first few Bernoulli numbers are $ B_0 = 1 $, $ B_1 = -1/2 $, $ B_2 = 1/6 $ and $ B_4 = -1/30 $. It is well-known that $ B_{2n + 1} = 0 $ for all $ n \geq 1 $. 

In recent years, considerable effort has been devoted to establishing connections between Bernoulli numbers and harmonic numbers $ H_n $, defined by
\begin{align*}
	H_n := \sum_{k = 1}^n \frac{1}{k},
\end{align*}
with the convention that $ H_0 = 0 $. To describe some of these identities, we introduce the Stirling numbers of the first and second kinds, defined, respectively, by
\begin{align*}
	x^{(n)} = x(x + 1) \cdots (x + n - 1) = \sum_{k = 0}^n \begin{bmatrix} n \\ k \end{bmatrix} x^k,
\end{align*}
and 
\begin{align*}
	x^n = \sum_{k = 0}^n (-1)^{n - k} \begin{Bmatrix} n \\ k \end{Bmatrix} x^{(k)}.
\end{align*}
Various properties of these number can be found in the book by Graham, Knuth, and Patashnik \cite{Graham_Knuth_Patashnik}. Several of these properties will be used throughout this paper.  

A simple identity connecting Bernoulli numbers, harmonic numbers, and Stirling numbers was obtained by Kellner \cite{Kellner} using differentiation
\begin{align*}
	\sum_{k = 1}^n (-1)^{k - 1} \begin{Bmatrix} n \\ k \end{Bmatrix} \frac{k!}{k + 1} H_k = \frac{n}{2} B_{n - 1}.
\end{align*}

More recently, Akkanat and Kargin \cite{Akkanat_Kargin} established two related identities
\begin{align*}
	&\sum_{k = 1}^n (-1)^{k - 1} \begin{Bmatrix} n \\ k \end{Bmatrix} (k - 1)! H_k = B_{n - 1}, \\
	&\sum_{k = 2}^n (-1)^{k - 1} \begin{Bmatrix} n \\ k \end{Bmatrix} (k - 1)! H_k H_{k - 1} = \frac{n + 1}{2} B_{n - 2}. 
\end{align*}
Observe that $ (k - 1)! = \begin{bmatrix} k \\ 1 \end{bmatrix} $ and $ (k - 1)! H_{k - 1} = \begin{bmatrix} k \\ 2 \end{bmatrix} $. These formulas suggest a more general formula involving $ \begin{bmatrix} k \\ j \end{bmatrix} $ for each $ j $. Indeed, Kargin and Mutluer \cite{Kargin_Mutluer} proved that
\begin{align*}
	\sum_{k = j}^n (-1)^{k - j} \begin{Bmatrix} n \\ k \end{Bmatrix} \begin{bmatrix} k \\ j \end{bmatrix} H_k = \bigg[ \binom{n}{j} - 1 \bigg] \frac{B_{n - j}}{n - j}.
\end{align*}
In this paper, we build upon the method of Kargin and Mutluer to derive several further identities involving these numbers.

\section{Main Result}

We now state the main result, which will be used throughout this paper.
\begin{theorem} \label{Th: Main Theorem} Let $ f(x) $ be a real-valued function. Define 
\begin{align*}
	\mathcal{S}_{n, j} (f) := \sum_{k = j}^n (-1)^{k - j} \begin{Bmatrix} n \\ k \end{Bmatrix} \begin{bmatrix} k \\ j \end{bmatrix} f(k).
\end{align*}
Then, for $ 1 \leq j \leq n $, we have
\begin{align*}
	\mathcal{S}_{n + 1, j} (f) - \mathcal{S}_{n, j - 1} (f) = \mathcal{S}_{n + 1, j} (\Delta f) - \mathcal{S}_{n, j} (x \Delta f),
\end{align*}
where $ \Delta f(k) = f(k) - f(k - 1) $ denotes the backward difference of $ f $.
\end{theorem}

\begin{proof}
We begin by considering the first term on the left-hand side
\begin{align*}
	\mathcal{S}_{n + 1, j} (f) = \sum_{k = j}^{n + 1} (-1)^{k - j} \begin{Bmatrix} n + 1 \\ k \end{Bmatrix} \begin{bmatrix} k \\ j \end{bmatrix} f(k).
\end{align*}
Applying the recurrence relation for Stirling numbers of the second kind \cite[pg. 164]{Graham_Knuth_Patashnik}
\begin{align*}
	\begin{Bmatrix} n + 1 \\ k \end{Bmatrix} = k \begin{Bmatrix} n \\ k \end{Bmatrix} + \begin{Bmatrix} n \\ k - 1 \end{Bmatrix},
\end{align*}
we see that 
\begin{align*}
	\mathcal{S}_{n + 1, j} (f) = \sum_{k = j}^{n + 1} (-1)^{k - j} \begin{Bmatrix} n \\ k \end{Bmatrix} \begin{bmatrix} k \\ j \end{bmatrix} kf(k) + \sum_{k = j}^{n + 1} (-1)^{k - j} \begin{Bmatrix} n \\ k - 1 \end{Bmatrix} \begin{bmatrix} k \\ j \end{bmatrix} f(k).
\end{align*}
We also use the recurrence relation for Stirling numbers of the first kind \cite[pg. 164]{Graham_Knuth_Patashnik}
\begin{align*}
	\begin{bmatrix} k + 1 \\ j \end{bmatrix} = k \begin{bmatrix} k \\ j \end{bmatrix} + \begin{bmatrix} k \\ j - 1 \end{bmatrix},
\end{align*}
to obtain 
\begin{multline} \label{Eq: 1}
	\mathcal{S}_{n + 1, j} (f) = \sum_{k = j}^n (-1)^{k - j} \begin{Bmatrix} n \\ k \end{Bmatrix} \begin{bmatrix} k + 1 \\ j \end{bmatrix} f(k)  - \sum_{k = j}^n (-1)^{k - j} \begin{Bmatrix} n \\ k \end{Bmatrix} \begin{bmatrix} k \\ j - 1 \end{bmatrix} f(k) - \\
	\sum_{k = j - 1}^n (-1)^{k - j} \begin{Bmatrix} n \\ k \end{Bmatrix} \begin{bmatrix} k + 1 \\ j \end{bmatrix} f(k + 1).
\end{multline}
Separating the term corresponding to $ k = j - 1 $ in the last sum gives
\begin{align*}
	\sum_{k = j - 1}^n (-1)^{k - j} \begin{Bmatrix} n \\ k \end{Bmatrix} \begin{bmatrix} k + 1 \\ j \end{bmatrix} f(k + 1) = - \begin{Bmatrix} n \\ j - 1 \end{Bmatrix} f(j) + \sum_{k = j}^n (-1)^{k - j} \begin{Bmatrix} n \\ k \end{Bmatrix} \begin{bmatrix} k + 1 \\ j \end{bmatrix} f(k + 1).
\end{align*}
The first term on the right hand side can be combined with the second sum in Equation \eqref{Eq: 1} to give $ \mathcal{S}_{n, j - 1} (f) $. Hence, Equation \eqref{Eq: 1} simplifies to 
\begin{align*}
	\mathcal{S}_{n + 1, j} (f) = \sum_{k = j - 1}^n (-1)^{k - j} \begin{Bmatrix} n \\ k \end{Bmatrix} \begin{bmatrix} k + 1 \\ j \end{bmatrix} (f(k) - f(k + 1)) + \mathcal{S}_{n, j - 1} (f),
\end{align*}
or equivalently
\begin{align*}
	\mathcal{S}_{n + 1, j} (f) = \sum_{k = j}^{n + 1} (-1)^{k - j} \begin{Bmatrix} n \\ k - 1 \end{Bmatrix} \begin{bmatrix} k \\ j \end{bmatrix} \Delta f (k) + \mathcal{S}_{n, j - 1} (f).
\end{align*}

Applying the recurrence relation for Stirling numbers of the second kind once more yields
\begin{align*}
	\mathcal{S}_{n + 1, j} (f) - \mathcal{S}_{n, j - 1} (f) = \sum_{k = j}^{n + 1} (-1)^{k - j} \begin{Bmatrix} n + 1 \\ k \end{Bmatrix} \begin{bmatrix} k \\ j \end{bmatrix} \Delta f(k) - \sum_{k = j}^n (-1)^{k - j} \begin{Bmatrix} n \\ k \end{Bmatrix} \begin{bmatrix} k \\ j \end{bmatrix} k \Delta f(k),
\end{align*}
which is precisely the required expression.
\end{proof}
	
\section{Applications of Main Result}

\subsection{Example 1}

When $ f(k) = H_k $, we have $ \Delta f(k) = H_k - H_{k - 1} = 1/k $. The right hand side of Theorem \ref{Th: Main Theorem} becomes
\begin{align*}
	\sum_{k = j}^{n + 1} (-1)^{k - j} \begin{Bmatrix} n + 1 \\ k \end{Bmatrix} \begin{bmatrix} k \\ j \end{bmatrix} \frac{1}{k} - \sum_{k = j}^n (-1)^{k - j} \begin{Bmatrix} n \\ k \end{Bmatrix} \begin{bmatrix} k \\ j \end{bmatrix}.
\end{align*}
Using the identities \cite[Eq. (6.100)]{Graham_Knuth_Patashnik}
\begin{align*}
	\sum_{k = j}^n (-1)^{k - j} \begin{Bmatrix} n \\ k \end{Bmatrix} \begin{bmatrix} k \\ j \end{bmatrix} \frac{1}{k} = \binom{n - 1}{j} \frac{B_{n - j}}{n- j},
\end{align*}
and \cite[pg. 164]{Graham_Knuth_Patashnik}	
\begin{align*}
	\sum_{k = j}^n (-1)^{k - j} \begin{Bmatrix} n \\ k \end{Bmatrix} \begin{bmatrix} k \\ j \end{bmatrix} = \delta_{n, j},
\end{align*}
where $ \delta_{m, n} $ denotes the Kronecker delta, we get for $ n > j $, 
 \begin{align*}
 	\mathcal{S}_{n + 1, j} (H_k) - \mathcal{S}_{n, j - 1} (H_k) = \binom{n}{j} \frac{B_{n + 1 - j}}{n + 1 - j}.
 \end{align*}

We can iterate this identity by successively lowering the value of $ j $. Since $ \mathcal{S}_{n - j, 0} (H_k) = 0 $, we obtain
\begin{align*}
	\mathcal{S}_{n + 1, j} (H_k) = \frac{B_{n + 1 - j}}{n + 1 - j} \sum_{i = 0}^{j - 1} \binom{n - i}{j - i} = \bigg[ \binom{n + 1}{j} - 1 \bigg] \frac{B_{n + 1 - j}}{n + 1 - j}.
\end{align*}
Replacing $ n + 1 $ by $ n $, we arrive at
\begin{align} \label{Eq: 2}
	\sum_{k = j}^n (-1)^{k - j} \begin{Bmatrix} n \\ k \end{Bmatrix} \begin{bmatrix} k \\ j \end{bmatrix} H_k = \bigg[ \binom{n}{j} - 1 \bigg] \frac{B_{n - j}}{n - j},
\end{align}
which recovers the main theorem of Kargin and Mutluer \cite{Kargin_Mutluer}.

\subsection{Example 2}

When $ f(k) = kH_k $, we have $ \Delta f(k) = kH_k - kH_{k - 1} + H_{k - 1} = H_{k - 1} + 1 = H_k - 1/k + 1 $. Therefore, Theorem \ref{Th: Main Theorem} gives
\begin{multline*}
	\mathcal{S}_{n + 1, j} (kH_k) - \mathcal{S}_{n, j - 1} (kH_k) = \sum_{k = j}^{n + 1} (-1)^{k - j} \begin{Bmatrix} n + 1 \\ k \end{Bmatrix} \begin{bmatrix} k \\ j \end{bmatrix} \bigg( H_k - \frac{1}{k} + 1 \bigg) \\
	- \sum_{k = j}^n (-1)^{k - j} \begin{Bmatrix} n \\ k \end{Bmatrix} \begin{bmatrix} k \\ j \end{bmatrix} (kH_k + k - 1).
\end{multline*}
As $ \mathcal{S}_{n, j} (1) = \delta_{n, j} $, the constant terms in both sums vanishes for $ n > j $. Thus, the remaining terms become
\begin{multline} \label{Eq: 3}
	\mathcal{S}_{n + 1, j} (kH_k) + \mathcal{S}_{n, j} (kH_k) - \mathcal{S}_{n, j - 1} (kH_k) = \sum_{k = j}^{n + 1} (-1)^{k - j} \begin{Bmatrix} n + 1 \\ k \end{Bmatrix} \begin{bmatrix} k \\ j \end{bmatrix} \bigg( H_k - \frac{1}{k} \bigg) \\
	- \sum_{k = j}^n (-1)^{k - j} \begin{Bmatrix} n \\ k \end{Bmatrix} \begin{bmatrix} k \\ j \end{bmatrix} k.
\end{multline}

The first sum can be evaluated using the identities established above. We therefore need only to look at the second sum. The following lemma will be useful.
\begin{lemma} \label{Th: Lemma 2} For $ 1 \leq j \leq n $, we have
\begin{align*}
	\sum_{k = j}^n (-1)^{k - j} \begin{Bmatrix} n \\ k \end{Bmatrix} \begin{bmatrix} k \\ j \end{bmatrix} k = (-1)^{n - j} \binom{n}{j - 1}.
\end{align*}
\end{lemma}

\begin{proof}
The exponential generating function for Stirling numbers of the first kind \cite[pg. 351]{Graham_Knuth_Patashnik} is  
\begin{align*}
	\sum_{k = j}^\infty (-1)^{k - j} \begin{bmatrix} k \\ j \end{bmatrix} \frac{z^k}{k!} = \frac{(\log (1 + z))^j}{j!}.
\end{align*}
Differentiating both sides with respect to $ z $ and multiplying by $ z $, we get
\begin{align*}
	\sum_{k = j}^\infty (-1)^{k - j} \begin{bmatrix} k \\ j \end{bmatrix} \frac{k z^k}{k!} = \frac{z}{1 + z} \frac{(\log (1 + z))^{j - 1}}{(j - 1)!}. 
\end{align*}
Applying the Stirling transformation amounts to replacing $ z $ with $ e^z - 1 $ on the right hand side. Hence, we obtain 
\begin{align*}
	\sum_{n = 1}^\infty \bigg( \sum_{k = j}^n (-1)^{k - j} \begin{Bmatrix} n \\ k \end{Bmatrix} \begin{bmatrix} k \\ j \end{bmatrix} k \bigg) \frac{z^n}{n!} = \frac{e^z - 1}{e^z} \frac{(\log (e^z))^{j - 1}}{(j - 1)!} = (1 - e^{-z}) \frac{z^{j - 1}}{(j - 1)!}.
\end{align*}
On the other hand, the Taylor expansion of $ e^{-z} $ gives
\begin{align*}
	(1 - e^{-z}) \frac{z^{j - 1}}{(j - 1)!} = \sum_{k = 1}^\infty (-1)^{k - 1} \frac{z^{k + j - 1}}{k! (j - 1)!} = \sum_{n = 1}^\infty (-1)^{n - j} \binom{n}{j - 1} \frac{z^n}{n!},
\end{align*}
where, in the second equality, we set $ n = k + j - 1 $. Comparing coefficients of $ z^n/n! $ yields
\begin{align*}
	\sum_{k = j}^n (-1)^{k - j} \begin{Bmatrix} n \\ k \end{Bmatrix} \begin{bmatrix} k \\ j \end{bmatrix} k = (-1)^{n - j} \binom{n}{j - 1},
\end{align*}
as required.
\end{proof}
	
Returning to the Equation \eqref{Eq: 3}, we have
\begin{align} \label{Eq: 4}
	\mathcal{S}_{n + 1, j} (kH_k) + \mathcal{S}_{n, j} (kH_k) - \mathcal{S}_{n, j - 1} (kH_k) = \bigg[ \binom{n}{j - 1} - 1 \bigg] \frac{B_{n + 1 - j}}{n + 1 - j} - (-1)^{n - j} \binom{n}{j - 1}.
\end{align}

Equation \eqref{Eq: 4} provides a recurrence relation for evaluating $ \mathcal{S}_{n, j} (kH_k) $. We first consider the case $ j = 1 $. Since $ \mathcal{S}_{n, 0} (kH_k) = 0 $ for all $ n \geq 1 $, Equation \eqref{Eq: 4} reduces to 
\begin{align*}
	\mathcal{S}_{n + 1, 1} (kH_k) + \mathcal{S}_{n, 1}  (kH_k) = (-1)^n.
\end{align*}
This leads to the following result.
\begin{theorem} \label{Th: Theorem 3} For $ n \geq 1 $, it holds that
\begin{align*}
	\mathcal{S}_{n, 1} (kH_k) = \sum_{k = 1}^n (-1)^{k - 1} \begin{Bmatrix} n \\ k \end{Bmatrix} k! H_k = (-1)^{n - 1} n.
\end{align*}
\end{theorem}

\begin{proof}
We proceed by induction on $ n $. For $ n = 1 $, we have $ \mathcal{S}_{1, 1} (kH_k) = 1 $, so the result holds in the base case. Suppose that $ \mathcal{S}_{m, 1} (kH_k) = (-1)^{m - 1} m $ for some $ m \geq 1 $. Taking $ n = m + 1 $, we obtain
\begin{align*}
	\mathcal{S}_{m + 1, 1} (kH_k) = (-1)^m - \mathcal{S}_{m, 1} (kH_k) =  (-1)^n - (-1)^{m - 1} m = (-1)^m (m + 1).
\end{align*}
This completes the induction step and the proof.
\end{proof} 
This identity also follows directly by applying the Stirling transformation to the well-known identity
\begin{align*}
n! H_n = \sum_{k = 1}^n k\begin{bmatrix} n \\ k \end{bmatrix}.
\end{align*}

Next, we consider the case $ j = 2 $. Using the fact that $ \mathcal{S}_{n, 1} (kH_k) = (-1)^{n - 1} n $, Equation \eqref{Eq: 4} becomes
\begin{align*}
	\mathcal{S}_{n + 1, 2} (kH_k) + \mathcal{S}_{n, 2} (kH_k) = (-1)^{n - 1} (2n) + B_{n - 1}.
\end{align*}
This recurrence yields the following theorem.
\begin{theorem} \label{Th: Theorem 4} For $ n \geq 2 $, it holds that 
\begin{align*}
	\mathcal{S}_{n, 2} (kH_k) = \sum_{k = 2}^n (-1)^k \begin{Bmatrix} n \\ k \end{Bmatrix} k! H_k H_{k - 1} = (-1)^n n(n - 1) + \sum_{k = 0}^{n - 2} (-1)^{n - k} B_k.
\end{align*}
\end{theorem}

\begin{proof}
The first equality follows from the fact that $ \begin{bmatrix} k \\ 2 \end{bmatrix} = (k - 1)! H_{k - 1} $. We prove the second equality by induction. For $ n = 2 $, we have $ \mathcal{S}_{2, 2} (kH_k) = 2H_2 = 3 = 2 + B_0 $, so the base case is done. Suppose that the equality holds for some $ m \geq 2 $. Taking $ n = m + 1 $, we get
\begin{align*}
	\mathcal{S}_{m + 1, 2} (kH_k) &= (-1)^{m - 1} m(m - 1) - \sum_{k = 0}^{m - 2} (-1)^{m - k} B_k + (-1)^{m - 1} (2m) + B_{m - 1} \\
	&= (-1)^m m(m + 1) + \sum_{k = 0}^{m - 1} (-1)^{m + 1 - k} B_k.
\end{align*}
Thus, the result holds for $ n = m + 1 $, completing the induction.
\end{proof} 

Notice that $ B_{2n + 1} = 0 $ for all $ n \geq 1 $, so we can further simplify the sum on the right hand side of Theorem \ref{Th: Theorem 4}. Suppose that $ n \geq 3 $, then
\begin{align*}
	\sum_{k = 0}^{n - 2} (-1)^k B_k = 1 + \sum_{k = 0}^{n - 2} B_k,
\end{align*}
where the extra $ 1 $ comes from replacing $ -B_1 = 1/2 $ in the sum with $ B_1 = -1/2 $. Therefore, by adding an extra $ k = 1 $ term in the sum, which does not change its value, we can rewrite Theorem \ref{Th: Theorem 4} as
\begin{align*}
	\sum_{k = 1}^n (-1)^{n - k} \begin{Bmatrix} n \\ k \end{Bmatrix} k! H_k H_{k - 1} = n(n - 1) + 1 + \sum_{k = 0}^{n - 2} B_k.
\end{align*}
Since $ H_{k - 1} = H_k - 1/k $, the left hand side becomes	 
\begin{align*}
	\sum_{k = 1}^n (-1)^{n - k} \begin{Bmatrix} n \\ k \end{Bmatrix} k! H_k^2 - \sum_{k = 1}^n (-1)^{n - k} \begin{Bmatrix} n \\ k \end{Bmatrix} (k - 1)! H_k = n(n - 1) + 1 + \sum_{k = 0}^{n - 2} B_k.
\end{align*}
Using Equation \eqref{Eq: 2} with $ j = 1 $, the second sum on the left hand side equals to $ B_{n - 1} $. Therefore, we arrive at the next corollary.
\begin{corollary} For $ n \geq 3 $, it holds that 
\begin{align*}
	\sum_{k = 1}^n (-1)^{n - k} \begin{Bmatrix} n \\ k \end{Bmatrix} k! H_k^2 = n(n - 1) - 1 + \sum_{k = 0}^{n - 1} B_k.
\end{align*}
\end{corollary}



\end{document}